\documentclass[11pt]{amsart}

\usepackage{amsmath,amssymb,txfonts,amsthm}
\usepackage[pagewise]{lineno}

\newtheorem{thm}{Theorem}[section]

\newtheorem{prop}[thm]{Proposition}
\newtheorem{lem}[thm]{Lemma}

\theoremstyle{definition}

\newtheorem{defi}[thm]{Definition}
\newtheorem{rem}[thm]{Remark}
\newtheorem{prob}[thm]{Problem}

\begin{document}

\author{Morimichi Kawasaki}
\address[Morimichi Kawasaki]{Center for Geometry and Physics, Institute for Basic Science (IBS), Pohang 37673, Republic of Korea}
\email{kawasaki@ibs.re.kr}
\title{Bavard's duality theorem on conjugation-invariant norms}
\maketitle
\begin{abstract}
Bavard proved a duality theorem between commutator length and quasimorphisms.
Burago, Ivanov and Polterovich  introduced the notion of a conjugation-invariant norm which is a generalization of commutator length.
Entov and Polterovich proved that Oh-Schwarz spectral invariants are subset-controlled quasimorphisms which are geralizations of quasimorphisms.
In the present paper, we prove a Bavard-type duality theorem between conjugation-invariant (pseudo-)norms and subset-controlled quasimorphisms on stable groups.
We also pose a generalization of our main theorem and prove that ``stably non-displaceable subsets of symplectic manifolds are heavy" in a rough sense if that generalization holds.
\end{abstract}


\section{Definitions and Results}
\subsection{Definition}
Burago, Ivanov and Polterovich defined the notion of conjugation-invariant (pseudo-) norms on groups in \cite{BIP} and they gave a number of its applications. 
\begin{defi}[{\cite{BIP}}]\label{cin}
Let $G$ be a group. A function $\nu\colon G\to \mathbb{R}_{\geq0}$ is \textit{ a conjugation-invariant norm} on $G$ if $\nu$ satisfies the following axioms:
\begin{itemize}
\item[(1)] $\nu(1)=0$;
\item[(2)] $\nu(f)=\nu(f^{-1})$ for every $f\in G$;
\item[(3)] $\nu(fg)\leq \nu(f)+\nu(g)$ for every $f,g\in G$;
\item[(4)] $\nu(f)=\nu(gfg^{-1})$ for every $f,g\in G$;
\item[(5)] $\nu(f)>0$ for every $f\neq 1\in G$.
\end{itemize}
A function $\nu\colon G\to \mathbb{R}$ is \textit{a conjugation-invariant pseudo-norm} on $G$ if $\nu$ satisfies the above axioms (1),(2),(3) and (4).
\end{defi}
For a conjugation-invariant pseudo-norm $\nu$, let $s\nu$ denote the stabilization of $\nu$ \textit{i.e.} $s\nu(g)=\lim_{n\to \infty}\frac{\nu(g^n)}{n}$
(this limit exists by Fekete's Lemma).

For a perfect group $G$, the commutator length $cl$ on $G$ is a conjugation-invariant norm.
Bavard \cite{Bav} proved the following famous theorem (\cite{Bav}, see also \cite{C}).
\begin{thm}[Corollary of Bavard's duality theorem \cite{Bav}]
Let $g$ be an element of a perfect group $G$.
Then
$scl(g)>0$ if and only if there exists a homogeneous quasimorphism $\phi$ such that $\phi(g)>0$.
\end{thm}
For interesting applications of Bavard's duality theorem, see \cite{CMS}, \cite{EK} and \cite{M} for example.
After Bavard's work, Calegari and Zhuang \cite{CZ} proved a Bavard-type duality theorem on W-length which is also conjugation-invariant.
In the present paper, we give a Bavard-type duality theorem on general conjugation-invariant (pseudo-)norms for some groups which are stable in some sense.

To state our main theorem, we introduce the notion of  subset-controlled quasimorphism (partial quasimorphism, pre-quasimorphism) which is a generalization of quasimorphism.
\begin{defi}\label{fragmentation norm}
Let $G$ be a group and $H$ a subgroup of $G$.
We define the fragmentation norm $\nu_H$ with respect to $H$ by for an element $f$ of $G$,
 \[\nu_H(f)=\min\{k;\exists g_1\ldots,g_k\in G, \exists h_1,\ldots h_k\in H\textit{ such that }f=g_1h_1g_1^{-1}\cdots g_kh_kg_k^{-1}\}.\]
If there is no such decomposition of $f$ as above, we put $\nu_H(f)=\infty$.
\end{defi}

\begin{defi}\label{qm rt nu}
Let $H$ be a subgroup of a group $G$.
A function $\phi\colon G\to\mathbb{R}$ is called an \textit{$H$-quasimorphism} if there exists a positive number $C$ such that for any elements $f$, $g$ of $G$,
\[|\phi(fg)-\phi(f)-\phi(g)|<C\,\min\{\nu_H(f),\nu_H(g)\}.\]
The infimum of such $C$ is called \textit{the defect of} $\phi$ and let $D(\phi)$ denote the defect of $\phi$.
$\phi$ is called \textit{homogeneous} if $\phi(f^n)=n\phi(f)$ for any element $f$ of $G$ and any  integer $n$.
\end{defi}
Such generalization of quasimorphisms appeared first in \cite{EP}.
Entov and Polterovich proved that Oh-Schwarz spectral invariants (for example, see \cite{Sc}. \cite{Oh}) are controlled quasimorphisms.

\begin{rem}\label{TetsuTomo}
In \cite{Ka}, $H$-quasimorphism is called quasimorphism relative to $\nu_H$.
Tomohiko Ishida and Tetsuya Ito pointed out that quasimorphism relative to $H$ usually means quasimorphism which vanishes on $H$ usually.
Thus we use a different notation from \cite{Ka}.
\end{rem}

Let $K$ be a subset of a group $G$.
For elements $f,g$ of $G$, let $fKg$ denote the subset $\{fkg;k\in K\}$ of $G$.
\begin{defi}
Let $H$ be a subgroup of a group $G$.
$G$ is \textit{c-generated} by $H$ if for any element $g$ of $G$, $\nu_H(g)<\infty$.
\end{defi}

The author essentially proved the following proposition.
\begin{prop}[\cite{Ka}]\label{previous result}
Let $G$ be a group c-generated by a perfect subgroup $H$ (in particular, $G$ is also perfect).
If there exists an $H$-quasimorphism $\phi$ with $\lim_{k\to\infty}\frac{\phi(g^k)}{k}>0$ for some $g$, 
then there exists a conjugation-invariant norm $\nu$ with $s\nu(g)>0$ (such a norm is called \textit{stably unbounded} \cite{BIP}).
\end{prop}

Our main theorem (Theorem \ref{main thm}) in the present paper is a converse of the Proposition \ref{previous result}.

\begin{rem}\label{KK}
The author \cite{Ka} proved that there exists such a $\mathrm{Ham}(\mathbb{B}^{2n})$-quasimorphism $\mu_K$ on $\mathrm{Ham}(\mathbb{R}^{2n})$.
Here, $\mathrm{Ham}(\mathbb{B}^{2n})$ and $\mathrm{Ham}(\mathbb{R}^{2n})$ are the group of Hamiltonian diffeomorphisms with compact support of the ball and the Euclidean space with the standard symplectic form, respectively.
He also proved that $\mu_K(g)>0$ for some commutator $g$ .
Thus, by Proposition \ref{previous result},  $[\mathrm{Ham}(\mathbb{R}^{2n}),\mathrm{Ham}(\mathbb{R}^{2n})]$ admits a stably unbounded norm.

Kimura \cite{Ki} proved a similar result on the infinite braid group $B_\infty=\cup_{k=1}^\infty B_k$ (the existence of a stably unbounded norm on $[B_\infty,B_\infty]$ is also proved by Brandenbursky and Kedra \cite{BK}).
\end{rem}


%
%

\begin{defi}
Let $G$ be a group, $H$ a subgroup of $G$ and $K$ a subset of $G$.
We define the set $\mathrm{D}^{f}_H(K)$ of maps \textit{displacing} $K$ \textit{far away} by
\begin{align*}
&\mathrm{D}^{f}_H(K)\\
&=\{h_0\in G; \forall g_1,\ldots,\forall g_k\in G, \exists h\in G\text{ such that }hh_0h^{-1}K(hh_0h^{-1})^{-1}\text{ commutes with }g_1Hg_1^{-1}\cup\cdots\cup g_kHg_k^{-1}\}.\\
\end{align*}
Let $\nu$ be a conjugation-invariant pseudo-norm on a group $G$.
For a subset $K$ of $G$, we define \textit{the far away displacement energy} $\mathrm{E}_{H,\nu}(K)$ of $K$ by
\[\mathrm{E}_{H,\nu}(K)=\inf_{g\in \mathrm{D}^{f}_H(K)}\nu(g).\]
\end{defi}
\begin{defi}
Let $G$ be a group and $H$ a subgroup of $G$.
$(G,H)$ satisfies the property $\mathsf{FM}$  if $G$ and $H$ satisfy the following conditions.
\begin{itemize}
\item[(1)] $G$ is c-generated by $H$,
\item[(2)] For any elements $h_1,\ldots,h_k$ of $G$, $\mathrm{D}^{f}_H(h_1Hh_1^{-1}\cup\cdots\cup h_kHh_k^{-1} )\neq\emptyset$.
\end{itemize}
A group $G$ satisfies the property $\mathsf{FM}$ if $(G,H)$ satisfies the property $\mathsf{FM}$ for some subgroup $H$.
\end{defi}

For a group $G$, we define the set $\mathsf{FM}(G)$ by
\[\mathsf{FM}(G)=\{H\le G ; \text{$(G,H)$ satisfies the property $\mathsf{FM}$}\}.\]
We give some examples satisfying the property $\mathsf{FM}$.
\begin{prop}\label{FM examples}

(1)The pairs $(B_\infty,B_i)$, $([B_\infty,B_\infty],[B_i,B_i])$ satisfy the property $\mathsf{FM}$ for any integer $i$.

(2)We consider the Riemannian surface $\Sigma_\infty=\cup_{k=1}^\infty\Sigma_k^1$ where $\Sigma_k^1$ is the Riemannian surface which has $k$ genus and $1$ puncture.
The pair of mapping class groups $(\mathrm{MCG}(\Sigma_\infty), \mathrm{MCG}(\Sigma_i^1))$ satisfies the property $\mathsf{FM}$ for any integer $i$.

(3)The pairs $(\mathrm{Ham}(\mathbb{R}^{2n}),\mathrm{Ham}(\mathbb{B}^{2n}))$, $([\mathrm{Ham}(\mathbb{R}^{2n}),\mathrm{Ham}(\mathbb{R}^{2n})],[\mathrm{Ham}(\mathbb{B}^{2n}),\mathrm{Ham}(\mathbb{B}^{2n})])$ satisfy the property $\mathsf{FM}$, respectively.

\end{prop}

Our main theorem is the following one.
\begin{thm}\label{main thm}
Let $G$ be a group satisfying the property $\mathsf{FM}$ and $\nu$ a conjugation-invariant pseudo-norm on $G$.
Then,

(1) For any element $g$ of $G$ such that $s\nu(g)>0$, there exists a function $\phi\colon G\to\mathbb{R}$ which is a homogeneous $H$-quasimorphism for any element $H$ of $\mathsf{FM}(G)$ such that $\phi(g)>0$.

(2) For any element $g$ of the commutator subgroup $[G,G]$ and any element $H$ of $\mathsf{FM}(G)$,
\[s\nu(g)\leq 8\sup_\phi\frac{\phi(g)\cdot E_{H,\nu}(H)}{D(\phi)},\]
where the supremum is taken over the set of homogeneous $H$-quasimorphisms $\phi\colon G\to\mathbb{R}$.
\end{thm}

In Section \ref{Proof of main theorem}, we construct the normed vector space $A_\nu$ and prove Theorem \ref{main thm} by applying the Hahn-Banach Theorem to $A_\nu$.
In Section \ref{norm}, we prove that $A_\nu$ is a normed vector space.
In Section \ref{example proof}, we prove Proposition \ref{FM examples}.
In Section \ref{back to sg}, we pose a generalization of Theorem \ref{main thm} (Problem \ref{generalize}) and give its application to symplectic geometry.
In that section, we prove that ``stably non-displaceable subsets of symplectic manifolds are heavy" in a very rough sense if the positive answer of Problem \ref{generalize} holds.



\subsection*{Acknowledgment}
The author would like to thank Professor Leonid Polterovich for telling him the importance of the paper by Calegari and Zhuang.
He also thanks his advisor Professor Takashi Tsuboi for his helpful advice.
He thanks Mitsuaki Kimura for the consecutive discussion and Tomohiko Ishida and Tetsuya Ito for giving him the advice to change the notion (Remark \ref{TetsuTomo}).
This work is supported by IBS-R003-D1, the Grant-in-Aid for Scientific Research [KAKENHI No. 25-6631], the Grant-in-Aid for JSPS fellows and  the Program for Leading Graduate Schools, MEXT, Japan.


\section{Proof of main theorem}\label{Proof of main theorem}
To construct controlled quasimorphisms by using the Hahn-Banach Theorem, we consider the normed vector space $A_\nu$ which we define here.
The idea of our construction comes from \cite{CZ}.

For a group $G$, we define the set $A_G=\coprod_{k=0}^\infty(G\times \mathbb{R})^k$.
We denote elements of $A_G$ by $g_1^{s_1}\cdots g_k^{s_k}$, where $g_1,\ldots.g_k$ are elements of $G$ and $s_1,\ldots,s_k$ are real numbers.

Let $\nu$ be a conjugation-invariant pseudo-norm on $G$.
We define the $\mathbb{R}_{\geq0}$-valued function $||\cdot||_\nu\colon A_G\to\mathbb{R}_{\geq 0}$ by
\[||g_1^{s_1}\cdots g_k^{s_k}||_\nu=\lim_{n\to\infty}\frac{1}{n}\cdot\nu(g_1^{[s_1n]}\ldots g_k^{[s_kn]}),\]
where $[\cdot]$ denotes the integer part.
For the trivial element $1\in(G\times\mathbb{R})^0$ of $A_G$, we define $||1||_\nu=0$.
\begin{prop}\label{limit}
Let $\nu$ be a conjugation-invariant pseudo-norm on a group $G$ satisfying the property $\mathsf{FM}$.
Then for any element $g_1^{s_1}\cdots g_k^{s_k}$ of $A_G$ , the above limit $||g_1^{s_1}\cdots g_k^{s_k}||_\nu$ exists. Thus $||\cdot||_\nu$ is well-defined.
\end{prop}
We prove Proposition \ref{limit} in Section \ref{norm}.
First, we define some operations on $A_G$.
For elements $\mathtt{g}=g_1^{s_1}\cdots g_k^{s_k}$, $\mathtt{h}=h_1^{t_1}\cdots h_l^{t_l}$ of $A_G$ and a real number $\lambda$, we define $\mathtt{g}\cdot\mathtt{h}$, $ \bar{\mathtt{g}}$ and $\mathtt{g}^{(\lambda)}$ by
\begin{align*}
\mathtt{g}\cdot\mathtt{h}&=g_1^{s_1}\cdots g_k^{s_k}h_1^{t_1}\cdots h_l^{t_l},\\
\bar{\mathtt{g}}&=g_k^{-s_k}\cdots g_1^{-s_1},\\
\mathtt{g}^{(\lambda)}&=g_1^{\lambda s_1}\cdots g_k^{\lambda s_k}.\\
\end{align*}
By the definition of conjugation-invariant pseudo-norms, we can confirm that the function $||\cdot||_\nu\colon A_G\to\mathbb{R}$ satisfies the following properties easily.
For any element $\mathtt{g},\mathtt{h}$ of $A_G$,
\begin{align*}
||\mathtt{g}\cdot\mathtt{h}||_\nu&\leq ||\mathtt{g}||_\nu+||\mathtt{h}||_\nu,\\
||\mathtt{h}\cdot\mathtt{g}\cdot\bar{\mathtt{h}}||_\nu&=||\mathtt{g}||_\nu,\\
||\bar{\mathtt{g}}||_\nu&=||\mathtt{g}||_\nu.\\
\end{align*}
We define the equivalence relation $\sim$ by
$\mathtt{g}\sim \mathtt{h}$ if and only if
$||\mathtt{g}\cdot\bar{\mathtt{h}}||_\nu=0$.
We denote the set $A_G/\sim$ by $A_\nu$ and the function $||\cdot||_\nu\colon A_G\to\mathbb{R}$ on $A_G$ induces the function  $||\cdot||_\nu\colon A_\nu\to\mathbb{R}$ on $A_\nu$.

In the present paper, we want to consider $A_\nu$ as a $\mathbb{R}$-vector space with the norm $||\cdot||_\nu$.
We define a sum operation, a inverse operation and $\mathbb{R}$-action on $A_\nu$.
For elements $\mathsf{g}=[\mathtt{g}]$, $\mathsf{h}=[\mathtt{h}]$ of $A_\nu$ and a real number $\lambda$, we define $\mathsf{g}+\mathsf{h}$ and $\lambda\mathsf{g}$ by
\begin{align*}
\mathsf{g}+\mathsf{h}&=[\mathtt{g}\cdot\mathtt{h}],\\
\lambda \mathsf{g}&=[\mathtt{g}^{(\lambda)}].\\
\end{align*}
\begin{prop}\label{well-defined operations}
Assume that $G$ satisfies the property $\mathsf{FM}$.
Then the above operations are well-defined.
\end{prop}
To use the Hahn-Banach theorem, we prove that $A_\nu$ is a normed vector space.
\begin{prop}\label{norm statement}
Assume that $G$ satisfies the property $\mathsf{FM}$.
Then $(A_\nu,||\cdot||_\nu)$ is a normed vector space with respect to the above operations.
\end{prop}
We prove Proposition \ref{well-defined operations} and \ref{norm statement} in Section \ref{norm}.

Let $G$ be a group and $\nu$  a conjugation-invariant pseudo-norm on $G$.
Let $L(G,\nu)$ denote the set of Lipschitz continuous (linear) homomorphisms from $A_\nu$ to $\mathbb{R}$.
By the Hahn-Banach Theorem, Proposition \ref{norm statement} implies the following proposition.

\begin{prop}\label{HahnBanach}
Assume that $G$ satisfies the property $\mathsf{FM}$.
Then, for any element $\mathsf{g}$ of $A_\nu$,
\[||\mathsf{g}||_\nu=\sup_{\hat\phi\in L(G,\nu)}\frac{\hat\phi(\mathsf{g})}{l(\hat\phi)},\]
where $l(\hat\phi)$ is the optimal Lipschitz constant of $\hat\phi$.
\end{prop}

For an element $\hat\phi$ of $L(G,\nu)$, we define the map $\phi\colon G\to\mathbb{R}$ by $\phi(g)=\hat\phi([g^1])$ for an element $g$ of $G$.

\begin{prop}\label{hom to controlled qm}

Let $H$ be an element of $\mathsf{FM}(G)$.
For any element $\hat\phi$ of $L(G,\nu)$, $\phi$ is a homogeneous $H$-quasimorphism.
Moreover, $D(\phi)\leq8l(\hat\phi)\cdot E_{H,\nu}(H)$.
\end{prop}
To prove Proposition \ref{hom to controlled qm}, we use the following lemmas.
\begin{lem}\label{fundamental}
Let $G$ be a group and $H,K$  subgroups of $G$.
Assume that $(G,H)$ satisfies the property $\mathsf{FM}$.
Then for any element $g$ of $G$ and any element $f$ of $K$, $\nu([g,f])\leq4E_{H,\nu}(K)$.
\end{lem}

\begin{proof}
Let $f,g$ and $h_0$ be elements of $K,G$ and $\mathrm{D}^{f}_H(K)$, respectively.
Since $G$ is c-generated by $H$ and the set $\{f,g\}$ is a finite set, there exist elements $h_1,\ldots,h_k$ of $G$ such that $f,g\in\langle h_1Hh_1^{-1}\ldots,h_kHh_k^{-1}\rangle$.

Then, by the definition of $\mathrm{D}^{f}_H(K)$, there exists an element $h$ of $G$ such that $(hh_0h^{-1})K(hh_0h^{-1})^{-1}$ commutes with $\langle h_1Hh_1^{-1}\ldots,h_kHh_k^{-1}\rangle$.
Since $f\in K$ and $f,g\in\langle h_1Hh_1^{-1}\ldots,h_kHh_k^{-1}\rangle$, $(hh_0h^{-1})f(hh_0h^{-1})^{-1}$ commutes with both of $f$ and $g$ and thus $[g,f]=[g,[f,hh_0h^{-1}]]$ holds.

Since $\nu$ is a conjugation-invariant pseudo-norm,
\begin{align*}
\nu([g,f])
& \leq \nu(g[f,hh_0h^{-1}]g^{-1})+\nu([f,hh_0h^{-1}]^{-1})\\
&=2\nu([f,hh_0h^{-1}])\\
      & \leq 2(\nu(f(hh_0h^{-1})f^{-1})+\nu((hh_0h^{-1})^{-1}))\\
      &=4\nu(hh_0h^{-1})=4\nu(h_0).
\end{align*}
By taking infimum,  $\nu([g,f])\leq4E_{H,\nu}(K)$.
\end{proof}

\begin{lem}[\cite{EP},\cite{Ki}]\label{EPK}
Let $G$ be a group, $H$ a subgroup of $G$ and $C$ a positive real number.
Assume that a map $\phi\colon G\to\mathbb{R}$ satisfies $|\phi(f)+\phi(g)-\phi(fg)|\leq C$ for any elements $f,g$ of $G$ with $\nu_H(f)=1$.
Then $\phi$ is an $H$-quasimorphism.
Moreover, $D(\phi)\leq 2C$.
\end{lem}

\begin{proof}[Proof of Proposition \ref{hom to controlled qm}]
Let $\hat\phi$ be an element of $L(G,\nu)$ and $f,g$ elements of $G$ with $\nu_H(f)=1$.
Since $H$ is a subgroup, $\nu_H(f^i)=1$ for any non-zero integer $i$.
Since $\nu$ is a conjugation-invariant pseudo-norm, by Lemma \ref{fundamental},
\begin{align*}
&|\phi(g)+\phi(f)-\phi(fg)|\\
&=|\hat\phi([g^1])+\hat\phi([f^1])-\hat\phi([(fg)^1])|\\
&=|\hat\phi([g^1]+[f^1]+(-1)[(fg)^1])|\\
&\leq l(\hat\phi)\cdot \lim_mm^{-1}\cdot\nu(g^mf^m(g^{-1}f^{-1})^m)\\
&= l(\hat\phi)\cdot \lim_mm^{-1}\cdot\nu((g^{m-1}[g,f^m]g^{-m+1})(g^{m-2}[g,f^{m-1}]g^{-m+2})\cdots(g^0[g,f]g^0))\\
&\leq  l(\hat\phi)\cdot \liminf_mm^{-1}\cdot\sum_{i=1}^{m-1}\nu([g,f^i])\\
&\leq  l(\hat\phi)\cdot \liminf_mm^{-1}\cdot (m-1)\cdot 4E_{H,\nu}(H)\\
&=4l(\hat\phi)\cdot E_{H,\nu}(H).\\
\end{align*}
Thus, by Lemma \ref{EPK}, $\phi$ is an $H$-quasimorphism and $D(\phi)\leq 8l(\hat\phi)\cdot E_{H,\nu}(H)$.
Since $\hat{\phi}\colon A_\nu\to\mathbb{R}$ is a homomorphism, $\phi\colon G\to\mathbb{R}$ is a homogeneous $H$-quasimorphism.
\end{proof}
\begin{proof}[Proof of Theorem \ref{main thm}]
Note that $||[g^1]||_\nu=s\nu(g)$ for any element $g$ of $G$.
Then (1) follows from Proposition \ref{HahnBanach} and \ref{hom to controlled qm}.
To prove (2), it is sufficient to prove it for an element $g$ of $[G,G]$ with $s\nu(g)>0$.
Then, by Proposition \ref{HahnBanach} and $||[g^1]||_\nu=s\nu(g)$, there exists an element $\hat\phi$ of $L(G,\nu)$ satisfying $\phi(g)=\hat\phi([g^1])\neq 0$.
Since $g\in[G,G]$, $D(\phi)>0$.
Thus Proposition \ref{hom to controlled qm} implies $8l(\hat\phi)^{-1}\leq D(\phi)^{-1}\cdot E_{H,\nu}(H)$.
Therefore Proposition \ref{HahnBanach} implies
\[s\nu(g)\leq 8\sup_\phi\frac{\phi(g)\cdot E_{H,\nu}(H)}{D(\phi)}.\]
\end{proof}


\section{Proof of being a normed vector space}\label{norm}
In the present section, we use the following notation.
\begin{defi}
Let $H$ be a subgroup of a group $G$ and $\nu$ a conjugation-invariant pseudo-norm on $G$.
For elements $g_1,\ldots,g_k$ of $G$, we define the far away displacement energy $E_{H,\nu}[g_1,\ldots,g_k]$ of $(g_1,\ldots,g_k)$ by  
\[E_{H,\nu}[g_1,\ldots,g_k]=\inf E_{H,\nu}( \langle h_1Hh_1^{-1},\ldots,h_lHh_l^{-1}\rangle),\]
where the infimum is taken over $h_1,\ldots,h_l$ such that $g_1,\ldots,g_k\in\langle h_1Hh_1^{-1},\ldots,h_lHh_l^{-1}\rangle$.
If $(G,H)$ satisfies the property $\mathsf{FM}$, $E_{H,\nu}[g_1,\ldots,g_k]<\infty$ for any elements $g_1,\ldots,g_k$ of $G$.
\end{defi}
To prove Proposition \ref{limit}, \ref{well-defined operations} and \ref{norm statement}, we use the following lemma.
\begin{lem}[\cite{CZ}]\label{continuity}
Let $G$ be a group and $\nu$ a conjugation-invariant pseudo-norm on $G$.
For any elements $g_1,\ldots,g_k$ of $G$ and integers $s_1,\ldots,s_k,t_1,\ldots,t_k$,
\[\nu((g_1^{s_1}\cdots g_k^{s_k})^{-1}(g_1^{t_1}\cdots g_k^{t_k}))\leq\sum_{i=1}^k|t_i-s_i|\cdot\nu(g_i).\]
\end{lem}
\begin{proof}
By using a graphical calculus argument (for example, see 2.2.4 of \cite{C}), there exist elements $h_1,\ldots,h_k$ of $\langle g_1,\cdots,g_k\rangle$ such that
\[(g_1^{s_1}\cdots g_k^{s_k})^{-1}(g_1^{t_1}\cdots g_k^{t_k})=h_k^{-1}g_k^{t_k-s_k}h_k\cdots h_1^{-1}g_1^{t_1-s_1}h_1.\]
Since $\nu$ is a conjugation-invariant pseudo-norm,
\[\nu((g_1^{s_1}\cdots g_k^{s_k})^{-1}(g_1^{t_1}\cdots g_k^{t_k}))\leq\sum_{i=1}^k\nu(h_i^{-1}g_i^{t_i-s_i}h_i)\leq\sum_{i=1}^k|t_i-s_i|\cdot\nu(g_i).\]
\end{proof}
\begin{proof}[Proof of Proposition \ref{limit}]
Fix an element $\mathsf{g}=[g_1^{s_1}\cdots g_k^{s_k}]$ of $A_\nu$.
Define a function $F\colon\mathbb{Z}_{>0}\to\mathbb{R}$ by $F(m)=\nu(g_1^{[s_1m]}\ldots g_k^{[s_km]})$.
By Fekete's Lemma, it is sufficient to prove that there exists a positive real number $C$ such that $F(m+n)\leq F(m)+F(n)+C$ for any positive integer $m,n$.
By Lemma \ref{continuity},
\begin{align*}
F(m+n)
&=\nu(g_1^{[s_1(m+n)]}\cdots g_k^{[s_k(m+n)]})\\
&\leq\nu(g_1^{[s_1m]+[s_1n]}\cdots g_k^{[s_km]+[s_kn]})+\nu((g_1^{[s_1m]+[s_1n]}\cdots g_k^{[s_km]+[s_kn]})^{-1}(g_1^{[s_1(m+n)]}\cdots g_k^{[s_k(m+n)]}))\\
&\leq\nu(g_1^{[s_1m]+[s_1n]}\cdots g_k^{[s_km]+[s_kn]})+\sum_{i=1}^k\nu(g_i).\\
\end{align*}
By using a graphical calculus argument, there exists an integer $l(k)$ which depends only on $k$ and elements $f_1,\ldots,f_{l(k)}$, $f_1^\prime,\ldots,f_{l(k)}^\prime$ of $\langle g_1,\ldots,g_k\rangle$ such that
\[(g_1^{[s_1m]}\cdots g_k^{[s_km]})^{-1}(g_1^{[s_1n]}\cdots g_k^{[s_kn]})^{-1}(g_1^{[s_1m]+[s_1n]}\cdots g_k^{[s_km]+[s_kn]})=[f_1,f_1^\prime]\cdots[f_{l(k)},f_{l(k)}^\prime].\]
Fix an element $H$ of $\mathsf{FM}(G)$.
Then $E_{H,\nu}[g_1,\ldots,g_k]<\infty$.
Thus, by Lemma \ref{fundamental},
\begin{align*}
&F(m+n)-F(m)-F(n)\\
&\leq\nu(g_1^{[s_1m]+[s_1n]}\cdots g_k^{[s_km]+[s_kn]})+\sum_{i=1}^k\nu(g_i)-\nu(g_1^{[s_1m]}\cdots g_k^{[s_km]})-\nu(g_1^{[s_1n]}\cdots g_k^{[s_kn]})\\
&\leq\nu((g_1^{[s_1m]}\cdots g_k^{[s_km]})^{-1}(g_1^{[s_1n]}\cdots g_k^{[s_kn]})^{-1}(g_1^{[s_1m]+[s_1n]}\cdots g_k^{[s_km]+[s_kn]}))+\sum_{i=1}^k\nu(g_i)\\
&\leq\nu([f_1,f_1^\prime]\cdots[f_{l(k)},f_{l(k)}^\prime])+\sum_{i=1}^k\nu(g_i)\\
&\leq\sum_{j=1}^{l(k)}\nu([f_j,f_j^\prime])+\sum_{i=1}^k\nu(g_i)\\
&\leq4l(k)E_{H,\nu}[g_1,\ldots,g_k]+\sum_{i=1}^k\nu(g_i).\\
\end{align*}
Thus we can apply Fekete's Lemma.
\end{proof}

To prove Proposition \ref{well-defined operations} and \ref{norm statement}, we use the following lemmas.
\begin{lem}\label{distribute}
Let $G$ be a group satisfying the property $\mathsf{FM}$ and $\nu$ any conjugation-invariant pseudo-norm on $G$.
Then for any element $\mathtt{g}$ of $A_G$ and any real numbers $\lambda_1,\lambda_2$, 
\[||\bar{\mathtt{g}}^{(\lambda_1+\lambda_2)}\cdot\mathtt{g}^{(\lambda_1)}\cdot\mathtt{g}^{(\lambda_2)}||_\nu=0.\]

\end{lem}
\begin{proof}
Assume that $\mathtt{g}$ are represented by $g_1^{s_1}g_2^{s_2}\cdots g_k^{s_k}\in A_G$.
For any integer $n$, by using a graphical calculus argument, there exist elements $f_{n,1},\ldots,f_{n,l(k)},f_{n,1}^\prime,\ldots,f_{n,l(k)}^\prime$ of $\langle g_1,\ldots,g_k\rangle$ such that
\begin{align*}
&(g_1^{[n\lambda_1s_1]+[n\lambda_2s_1]}g_2^{[n\lambda_1s_2]+[n\lambda_2s_2]}\cdots g_k^{[n\lambda_1s_k]+[n\lambda_2s_k]})^{-1}(g_1^{[n\lambda_1s_1]}g_2^{[n\lambda_1s_2]}\cdots g_k^{[n\lambda_1s_k]})(g_1^{[n\lambda_2s_1]}g_2^{[n\lambda_2s_2]}\cdots g_k^{[n\lambda_2s_k]})\\
&=[f_{n,1},f_{n,1}^\prime]\cdots[f_{n,l(k)},f_{n,l(k)}^\prime].\\
\end{align*}
Fix an element $H$ of $\mathsf{FM}(G)$.
Then $E_{H,\nu}[g_1,\ldots,g_k]<\infty$.
Thus, by  Lemma \ref{continuity} and Lemma \ref{fundamental},
\begin{align*}
&||\bar{\mathtt{g}}^{(\lambda_1+\lambda_2)}\cdot\mathtt{g}^{(\lambda_1)}\cdot\mathtt{g}^{(\lambda_2)}||_\nu\\
&=\lim_{n\to\infty}\frac{1}{n}\cdot\nu((g_1^{[n\lambda_1s_1+n\lambda_2s_1]}\cdots g_k^{[n\lambda_1s_k+n\lambda_2s_k]})^{-1}(g_1^{[n\lambda_1s_1]}\cdots g_k^{[n\lambda_1s_k]})(g_1^{[n\lambda_2s_1]}\cdots g_k^{[n\lambda_2s_k]}))\\
&\leq\lim_{n\to\infty}\frac{1}{n}\cdot(\nu((g_1^{[n\lambda_1s_1]+[n\lambda_2s_1]}\cdots g_k^{[n\lambda_1s_k]+[n\lambda_2s_k]})^{-1}(g_1^{[n\lambda_1s_1]}\cdots g_k^{[n\lambda_1s_k]})(g_1^{[n\lambda_2s_1]}\cdots g_k^{[n\lambda_2s_k]}))+\sum_{i=1}^k\nu(g_i))\\
&=\lim_{n\to\infty}\frac{1}{n}\cdot(\nu([f_{n,1},f_{n,1}^\prime]\cdots[f_{n,l(k)},f_{n,l(k)}^\prime])+\sum_{i=1}^k\nu(g_i))\\
&\leq\lim_{n\to\infty}\frac{1}{n}\cdot(\sum_{j=1}^{l(k)}\nu([f_{n,j},f_{n,j}^\prime])+\sum_{i=1}^k\nu(g_i))\\
&\leq\lim_{n\to\infty}\frac{1}{n}\cdot(4l(k)E_{H,\nu}[g_1,\ldots,g_k]+\sum_{i=1}^k\nu(g_i))\\
&=0.\\
\end{align*}
\end{proof}

\begin{lem}\label{real action}
Let $G$ be a group satisfying the property $\mathsf{FM}$ and $\nu$  a conjugation-invariant pseudo-norm on $G$.
For any elements $g_1,\ldots,g_k$ of $G$ and real numbers $\lambda,s_1,\ldots,s_k$,
\[\lim_{n\to\infty}\frac{1}{n}\cdot\nu(g_1^{[\lambda s_1n]}\cdots g_k^{[\lambda s_kn]})=|\lambda|\lim_{n\to\infty}\frac{1}{n}\cdot\nu(g_1^{[s_1n]}\cdots g_k^{[s_kn]}).\]
\end{lem}
\begin{proof}
We first prove for the case when $\lambda$ is a positive rational number \textit{i.e.} $\lambda=\frac{q}{p}$ where $p,q$ are positive integers.
By the existence of the limits (Proposition \ref{limit}), since the limit of any subsequence equals to the one of the original sequence,
\begin{align*}
&\lim_{n\to\infty}\frac{1}{n}\cdot\nu(g_1^{[\lambda s_1n]}\cdots g_k^{[\lambda s_kn]})\\
&=\lim_{n\to\infty}\frac{1}{pn}\cdot\nu(g_1^{[q s_1n]}\cdots g_k^{[q s_kn]})\\
&=\lim_{n\to\infty}\frac{q}{pn}\cdot\nu(g_1^{[s_1n]}\cdots g_k^{[s_kn]})\\
&=\lambda\lim_{n\to\infty}\frac{1}{n}\cdot\nu(g_1^{[s_1n]}\cdots g_k^{[s_kn]}).\\
\end{align*}
We prove for the case $\lambda=-1$.

Let $\mathtt{g}$ denote the element $g_1^{s_1}g_2^{s_2}\cdots g_k^{s_k}$ of $A_G$.
By Lemma \ref{distribute},
$[\mathtt{g}^{(-1)}\cdot\mathtt{g}]=[\mathtt{g}^{(0)}]=[1]$.
Recall that $1\in(G\times\mathbb{R})^0$ is the trivial element of $A_G$. 
Thus $[\mathtt{g}^{(-1)}]=[\mathtt{g}^{(-1)}\cdot\mathtt{g}\cdot\bar{\mathtt{g}}]=[1\cdot\bar{\mathtt{g}}]=[\bar{\mathtt{g}}]$.
Therefore $||(-1)\mathtt{g}||_\nu=||\bar{\mathtt{g}}||_\nu=||\mathtt{g}||_\nu$ and we complete the proof for the case when $\lambda$ is a rational number.

Since Lemma \ref{continuity} implies that the function $\mathbb{R}\to\mathbb{R}$, 
$\lambda\mapsto\lim_{n\to\infty}\frac{1}{n}\cdot\nu(g_1^{[\lambda s_1n]}\cdots g_k^{[\lambda s_kn]})$ is continuous, we complete the proof.
\end{proof}

\begin{proof}[Proof of Proposition \ref{well-defined operations}]
Assume that elements $\mathtt{f}_1,\mathtt{f}_2,\mathtt{g}_1,\mathtt{g}_2$ of $A_G$ satisfy $[\mathtt{f}_1]=[\mathtt{f}_2]$ and $[\mathtt{g}_1]=[\mathtt{g}_2]$.
Then
\begin{align*}
&||(\mathtt{f}_1\cdot\mathtt{g}_1)\cdot\overline{(\mathtt{f}_2\cdot\mathtt{g}_2)}||_\nu\\
&=||\mathtt{f}_1\cdot\mathtt{g}_1\cdot\bar{\mathtt{g}}_2\cdot\bar{\mathtt{f}}_2||_\nu\\
&\leq||\mathtt{f}_1\cdot\mathtt{g}_1\cdot\bar{\mathtt{g}}_2\cdot\bar{\mathtt{f}}_1||_\nu+||\mathtt{f}_1\cdot\bar{\mathtt{f}}_2||_\nu\\
&=||\mathtt{g}_1\cdot\bar{\mathtt{g}}_2||_\nu+||\mathtt{f}_1\cdot\bar{\mathtt{f}}_2||_\nu=0.\\
\end{align*}
Thus $[\mathtt{f}_1\cdot\mathtt{g}_1]=[\mathtt{f}_2\cdot\mathtt{g}_2]$.

Assume that elements $\mathtt{g}_1,\mathtt{g}_2$ of $A_G$ satisfy $[\mathtt{g}_1]=[\mathtt{g}_2]$.
For any real number $\lambda$, Lemma \ref{real action} implies 
\[||\bar{\mathtt{g}}_1^{(\lambda)}\cdot\mathtt{g_2}^{(\lambda)}||_\nu=||(\bar{\mathtt{g}}_1\cdot\mathtt{g}_2)^{(\lambda)}||_\nu=|\lambda|\cdot||(\bar{\mathtt{g}}_1\cdot\mathtt{g}_2)||_\nu=0.\]
Thus $[\mathtt{g}_1^{(\lambda)}]=[\mathtt{g}_2^{(\lambda)}]$.
\end{proof}
\begin{lem}\label{commute}
Let $G$ be a group satisfying the property $\mathsf{FM}$ and $\nu$  a conjugation-invariant pseudo-norm on $G$.
Then for any elements $\mathsf{f},\mathsf{g}$ of $A_\nu$, 
\[\mathsf{f}+\mathsf{g}=\mathsf{g}+\mathsf{f}.\]
\end{lem}
\begin{proof}
Assume that $\mathsf{f},\mathsf{g}$ are represented by $[\mathtt{f}]=[f_1^{s_1}f_2^{s_2}\cdots f_k^{s_k}]$, $[\mathtt{g}]=[g_1^{t_1}g_2^{t_2}\cdots g_l^{t_l}]$, respectively.
Fix an element $H$ of $\mathsf{FM}(G)$. Then $E_{H,\nu}[g_1,\ldots,g_l]<\infty$.
Since $g_1^{[t_1n]}g_2^{[t_2n]}\cdots g_l^{[t_ln]}\in\langle g_1,\ldots,g_l\rangle$ for any $n$, Lemma \ref{fundamental}  implies
\begin{align*}
&||\mathtt{f}\cdot\mathtt{g}\cdot\overline{(\mathtt{g}\cdot\mathtt{f})}||_\nu\\
&=||\mathtt{f}\cdot\mathtt{g}\cdot\bar{\mathtt{f}}\cdot\bar{\mathtt{g}}||_\nu\\
&=\lim_{n\to\infty}\frac{1}{n}\cdot\nu((f_1^{[s_1n]}f_2^{[s_2n]}\cdots f_k^{[s_kn]})(g_1^{[t_1n]}g_2^{[t_2n]}\cdots g_l^{[t_ln]})(f_1^{[s_1n]}f_2^{[s_2n]}\cdots f_k^{[s_kn]})^{-1}(g_1^{[t_1n]}g_2^{[t_2n]}\cdots g_l^{[t_ln]})^{-1})\\
&=\lim_{n\to\infty}\frac{1}{n}\cdot\nu([f_1^{[s_1n]}f_2^{[s_2n]}\cdots f_k^{[s_kn]},g_1^{[t_1n]}g_2^{[t_2n]}\cdots g_l^{[t_ln]}])\\
&=\lim_{n\to\infty}\frac{1}{n}\cdot 4E_{H,\nu}[g_1,\ldots,g_l]=0.\\
\end{align*}
Thus $\mathsf{f}+\mathsf{g}=[\mathtt{f}\cdot\mathtt{g}]=[\mathtt{g}\cdot\mathtt{f}]=\mathsf{g}+\mathsf{f}$.
\end{proof}

\begin{proof}[Proof of Proposition \ref{norm statement}]
By Lemma \ref{distribute}, \ref{real action} and \ref{commute}, for any elements $\mathsf{f},\mathsf{g}$ of $A_\nu$ and real numbers $\lambda_1,\lambda_2$,
\begin{itemize}
\item  $(\lambda_1+\lambda_2)\mathsf{g}=\lambda_1\mathsf{g}+\lambda_2\mathsf{g}$,
\item $||\lambda_1\mathsf{g}||_\nu=|\lambda_1|\cdot||\mathsf{g}||_\nu$,
\item $\mathsf{f}+\mathsf{g}=\mathsf{g}+\mathsf{f}$.
\end{itemize}
We can confirm the other axioms of a normed vector space easily. Thus we complete the proof of Proposition \ref{norm statement}.
\end{proof}


\section{Proof that examples satisfy the property $\mathsf{FM}$}\label{example proof}
In the present section, we prove that ($\mathrm{Ham}(\mathbb{R}^{2n}),\mathrm{Ham}(\mathbb{B}^{2n}))$ satisfies the property $\mathsf{FM}$.
We can prove other parts of Proposition \ref{FM examples} similarly.

We use the following notations.
For a diffeomorphism $g$ on a manifold $M$, let $\mathrm{Supp}(g)$ denote the support of $g$.
For a point $p$ of $\mathbb{R}^{2n}$ and a positive real number $R$, let $\mathbb{B}^{2n}(p,R)$ denote a subset $\{x\in\mathbb{R}^{2n} ; ||x-p||<R\}$ of $\mathbb{R}^{2n}$.
\begin{proof}
For simplicity, let $\mathcal{B}$ denote the subgroup $\mathrm{Ham}(\mathbb{B}^{2n})$ and $p_0$ denote the point $(3,0,\ldots,0)$ of $\mathbb{R}^{2n}$.

Let $f_0$ be a Hamiltonian diffeomorphism on $\mathbb{R}^{2n}$such that $f_0(\mathbb{B}^{2n})=\mathbb{B}^{2n}(p_0,1)$.
Fix Hamiltonian diffeomorphisms $g_1,\ldots,g_k$ on $\mathbb{R}^{2n}$.
Then there exists a positive real number $R$ such that $\mathrm{Supp}(g_1)\cup\cdots\cup\mathrm{Supp}(g_k)\subset\mathbb{B}^{2n}(0,R)$.
Since $f_0(\mathbb{B}^{2n})=\mathbb{B}^{2n}(p_0,1)$ and $\mathbb{B}^{2n}(p_0,1)\cap\mathbb{B}^{2n}=\emptyset$, we can take a Hamiltonian diffeomorphism $f$ such that $f(\mathbb{B}^{2n})=\mathbb{B}^{2n}$ and $ff_0(\mathbb{B}^{2n})\cap\mathbb{B}^{2n}(0,R)=\emptyset$.
Since $(ff_0f^{-1})\mathcal{B}(ff_0f^{-1})^{-1}=\mathrm{Ham}(ff_0f^{-1}(\mathbb{B}^{2n}))=\mathrm{Ham}(ff_0(\mathbb{B}^{2n}))$ and $g_1\mathcal{B}g_1^{-1}\cup\cdots\cup g_k\mathcal{B}g_k^{-1}=\mathrm{Ham(g_1(\mathbb{B}^{2n})\cup\cdots\cup g_k(\mathbb{B}^{2n}))}\subset \mathrm{Ham}(\mathbb{B}^{2n}(0,R))$, $ff_0(\mathbb{B}^{2n})\cap\mathbb{B}^{2n}(0,R)=\emptyset$ implies that $(ff_0f^{-1})\mathcal{B}(ff_0f^{-1})^{-1}$ commutes with $g_1\mathcal{B}g_1^{-1}\cup\cdots\cup g_k\mathcal{B}g_k^{-1}$.
Thus $f_0\in\mathrm{D}^{f}_{\mathcal{B}}(\mathcal{B})$.

Note that Banyaga's fragmentation lemma (\cite{Ba}) states that for any Hamiltonian diffeomorphism $g$, there exist Hamiltonian diffeomorphisms $f_1,\ldots,f_k$ such that $g\in \langle f_1\mathcal{B}f_1^{-1},\ldots,f_k\mathcal{B}f_k^{-1}\rangle$.
Thus $\mathrm{Ham}(\mathbb{R}^{2n})$ is c-generated by $\mathcal{B}$ and we complete the proof.
\end{proof}


\section{Are stably non-displaceable subsets heavy?: Bavard's duality in Hofer's geometry}\label{back to sg}
In the present paper, we have considered subgroups which are displaceable far away.
In the present section, we pose a problem on displaceable subgroups and give its application to symplectic geometry.

On notions related to symplectic geometry, we follow the ones of \cite{E}.
\begin{defi}
Let $G$ be a group, $H$ a subgroup of $G$ and $\mu\colon G\to\mathbb{R}$ an $H$-quasimorphism on $G$.
$\mu$ is called \textit{semi-homogeneous} if $\mu(g^n)=n\mu(g)$ for any element $g$ of $G$ and any non-negative integer $n$.
\end{defi}
Let $(M,\omega)$ be a $2m$-dimensional closed symplectic manifold.
A subset $X$ of $(M,\omega)$ is \textit{displaceable} if $\bar{X}\cap\phi_F^1(X)=\emptyset$ for some Hamiltonian function $F\colon S^1\times M\to\mathbb{R}$ where $\phi_F$ is the Hamiltonian diffeomorphism generated by $F$ and $\bar{X}$ is the topological closure of $X$.
$X$ is \textit{non-displaceable} otherwise.
Let $\mathsf{DO}(M)$ denote the set of displaceable open subsets of $(M,\omega)$.
A subset $X$ of a symplectic manifold $M$ is \textit{stably displaceable} if $X\times S^1$ is displaceable in $M\times T^\ast S^1$.
$X$ is \textit{stably non-displaceable} otherwise.

For an idempotent $a$ of the quantum homology $QH_\ast(M,\omega)$, Entov and Polterovich \cite{EP} defined the asymptotic spectral invariant $\mu_a\colon\widetilde{\mathrm{Ham}}(M)\to\mathbb{R}$ on the universal covering $\widetilde{\mathrm{Ham}}(M)$ of the group $\mathrm{Ham}(M)$ of Hamiltonian diffeomorphisms in terms of Oh-Schwarz spectral invariants and proved that $\mu_a$ is a semi-homogeneous $\widetilde{\mathrm{Ham}}_U(M)$-quasimorphism for any element $U$ of $\mathsf{DO}(M)$.
Here  $\widetilde{\mathrm{Ham}}_U(M)$ is the set of elements of $\widetilde{\mathrm{Ham}}(M)$ which are generated by Hamiltonian functions with support in $S^1\times U$.

A Hamiltonian function $F\colon S^1\times M\to\mathbb{R}$ is \textit{normalized} if $\int_MF_t\omega^m=0$ for any $t\in S^1$.
\begin{defi}[\cite{EP09}]
Let $(M,\omega)$ be a closed symplectic manifold and $a$ an idempotent of $QH_\ast(M,\omega)$.
A compact subset $X$ of $(M,\omega)$ is $a$-\textit{heavy} if for any normalized Hamiltonian function $F\colon S^1\times M\to\mathbb{R}$,
\[-\mu_a(\phi_F)\geq\mathrm{vol}(M)\cdot\inf_{S^1\times X}F,\]
where $\mathrm{vol}(M)=\int_M\omega^m$.

\end{defi}
In particular, if $X$ is $a$-heavy, $\mu_a(\phi_F)<0$ for any normalized Hamiltonian function $F$ with $F|_{S^1\times X}>0$.

\begin{rem}
The above definition of heaviness is different from the one of \cite{EP09} and \cite{E} (In their definition, they consider only autonomous Hamiltonian functions).
However, as remarked in \cite{S}, the above definition is known to be equivalent to the one of \cite{EP09} and \cite{E}.
\end{rem}

Entov and Polterovich also proved that heavy subsets are stably non-displaceable (\cite{EP09}).
In the present section, we consider the converse problem ``Are stably non-displaceable subsets heavy?".

\begin{defi}
Let $G$ be a group, $H$ a subgroup of $G$ and $K$ a subset of $G$.
We define the set $\mathrm{D}_H(K)$ of maps \textit{displacing} $K$ by
\[\mathrm{D}_H(K)=\{h_0\in G; h_0K(h_0)^{-1}\text{ commutes with }H\}.\]
\end{defi}

\begin{defi}
Let $G$ be a group and $H$ a subgroup of $G$.
$(G,H)$ satisfies the property $\mathsf{FD}$ if $G$ and $H$ satisfies the following conditions.
\begin{itemize}
\item[(1)] $G$ is c-generated by $H$,
\item[(2)] $\mathrm{D}_H(H)\neq\emptyset$.
\end{itemize}
A group $G$ satisfies the property $\mathsf{FD}$ if $(G,H)$ satisfies the property $\mathsf{FD}$ for some subgroup $H$.
\end{defi}

For a group $G$ which satisfies the property $\mathsf{FD}$, we define the set $\mathsf{FD}(G)$ by
\[\mathsf{FD}(G)=\{H\le G ; \text{$(G,H)$ satisfies the property $\mathsf{FD}$}\}.\]

We pose the following problem.
\begin{prob}\label{generalize}
Let $G$ be a group satisfying the property $\mathsf{FD}$ and $\nu$ a conjugation-invariant pseudo-norm on $G$.
Prove that
 for any element $g$ of $G$ such that $s\nu(g)>0$, there exists a function $\mu\colon G\to\mathbb{R}$ which is a semi-homogeneous $H$-quasimorphism for any element $H$ of $\mathsf{FD}(G)$ such that $\mu(g)>0$.
\end{prob}

Here, we give an application of Problem \ref{generalize} to symplectic geometry.

\begin{prop}\label{main proposition}
Assume that the positive answer of Problem \ref{generalize} holds.

Let $X$ be a stably non-displaceable compact subset of a closed symplectic manifold $(M,\omega)$.
For any normalized Hamiltonian function $F\colon S^1\times M\to\mathbb{R}$ with $F|_{S^1\times X}>0$,
there exists a function $\mu_F\colon\widetilde{\mathrm{Ham}}(M)\to\mathbb{R}$ which is a semi-homogeneous $\widetilde{\mathrm{Ham}}_U(M)$-quasimorphism for any element $U$ of $\mathsf{DO}(M)$ such that $\mu_F(\phi_F)<0$.
\end{prop}
Proposition \ref{main proposition} states that ``stably non-displaceable subsets are heavy" in a very rough sense if the positive answer of Problem \ref{generalize} holds.

To prove Proposition \ref{main proposition}, we use the following Polterovich's theorem.
\begin{thm}[\cite{P98},\cite{P01}]\label{stably non-disp and hofer}
Let $X$ be a stably non-displaceable subset of a closed symplectic manifold $(M,\omega)$.
For any normalized Hamiltonian function $F\colon S^1\times M\to\mathbb{R}$ with $F|_{S^1\times X}\geq p$ for some positive number $p$,
$||\phi_F||_H\geq p$.
Here $||\cdot||_H\colon\widetilde{\mathrm{Ham}}(M)\to\mathbb{R}$ is the Hofer norm which is known to be a conjugation-invariant pseudo-norm.
\end{thm}
\begin{proof}[Proof of Proposition \ref{main proposition}]
Since $X$ is compact, there exists some positive number $p$ with $F|_{S^1\times X}\geq p$.
For any positive integer $n$, we define a Hamiltonian function $F^{(n)}\colon S^1\times M\to\mathbb{R}$ by $F^{(n)}(t,x)=n\cdot F(nt,x)$.
Note that $\phi_{F^{(n)}}=(\phi_F)^n.$
Then, by $F^{(n)}|_{S^1\times X}\geq np$ and Theorem \ref{stably non-disp and hofer}, $||(\phi_F)^n||_H\geq np$ for any positive integer $n$.
Since $\widetilde{\mathrm{Ham}}_U(M)\in\mathsf{FD}(\widetilde{\mathrm{Ham}}(M))$ for any element $U$ of $\mathsf{DO}(M)$, by the positive answer of Problem \ref{generalize}, there exists a function $\mu_F^\prime\colon\widetilde{\mathrm{Ham}}(M)\to\mathbb{R}$ which is a semi-homogeneous $\widetilde{\mathrm{Ham}}_U(M)$-quasimorphism for any element $U$ of $\mathsf{DO}(M)$ such that $\mu_F^\prime(\phi_F)>0$.
Then $-\mu_F^\prime$ is a desired function. 
\end{proof}


\end{document}